\documentclass[11 pt]{article}
\usepackage{amsmath,amssymb,amsfonts,amsthm,graphicx,cite,epsfig,makecell, array,mathrsfs}
\usepackage[varg]{pxfonts}
\usepackage{wrapfig}
\usepackage{enumitem}
\usepackage{epstopdf,footnote,multirow,booktabs,bigints,pstool,booktabs,caption,subcaption ,lipsum}
\usepackage[table]{xcolor}
\usepackage[colorlinks = true,
            linkcolor = red,
            citecolor = blue]{hyperref}
\usepackage[euler]{textgreek}
\usepackage[symbol]{footmisc}
\usepackage{titlesec, float}
\usepackage[export]{adjustbox}
\theoremstyle{plain}
\newtheorem{theorem}{Theorem}[section]
\newtheorem{lemma}[theorem]{Lemma}

\theoremstyle{definition}

\newtheorem{corollary}[theorem]{Corollary}

\theoremstyle{remark}

\numberwithin{equation}{section}
\newcolumntype{?}{!{\vrule width 1.4pt}}

\def\correspondingauthor{\footnote{Corresponding author.}}
\titlelabel{\thetitle.\quad}
\renewenvironment{abstract}
               {\list{}{\rightmargin\leftmargin}%
                \item[\textbf{\hspace{8.6mm}Abstract ---}]\relax}
               {\endlist}
\DeclareUrlCommand{\url}{%
    \def\UrlLeft##1\UrlRight{\underline{##1}}}

\date{}
\title{An Application of Jackson's $(p, q)$-Derivative to a Subclass of Starlike Functions with Negative Coefficients}
\author{Feras Yousef$^{1,}$\correspondingauthor{}\,\,,\, Amal Al-Shible$^{2}$,\, Sibel Yal\c{c}in$^{3}$ \vspace{0.1in}\\
\footnotesize{$^{1,2}$Department of Mathematics, The University of Jordan, Amman 11942, Jordan }  \\
\footnotesize{$^{3}$Department of Mathematics, Faculty of Arts and Science, Uludag University, TR-16059 Bursa, Turkey}  \\
 \footnotesize{e-mail: $^1$fyousef@ju.edu.jo}, \,$^2$amalalshebli@yahoo.com, \,$^3$syalcin@uludag.edu.tr}
\begin{document}
\maketitle
%%%%%%%%%%%%%%%%%%%%%%%%%%%%%%%%%%%%%%%%%%%%%%%%%%%%%%%%%%%%%%%%%%%%%%%%%%%%%%%%%%%%%%%%%%%%%%%%%%%%%%%%%%%%%%%%%%%%%%%%%%%%%%%%%%%%%%%%%%%%%%%%%%%%%%%%%%%%%%%%%%%%%%%%%%%%%%%%%%%%%%%%%%%%%%%%%%%%%%%%
%%%%%%%%%%%%%%%%%%%%%%%%%%%%%%%%%%%%%%%%%%%%%%%%%%%%%%%%%%%%%%%%%%%%%%%%%%%%%%%%%%%%%%%%%%%%%%%%%%%%%%%%%%%%%%%%%%%%%%%%%%%%%%%%%%%%%%%%%%%%%%%%%%%%%%%%%%%%%%%%%%%%%%%%%%%%%%%%%%%%%%%%%%%%%%%%%%%%%%%%
                                %%%%%%%%%%%%%%%%%%%%%%%%%%%%%%%%%%%%%%%%%%%%%%%%%%%%%%%%%%%%%%%%%%%%%%%%%%%%%%%%%%%%%%%%%%%%%%%%%%%%%%%%%%%%%%%%%%%%%%%%%%%%%%%%%%%%%%%%
                                %%%%%%%%%%%%%%%%%%%%%%                                       Abstract                                         %%%%%%%%%%%%%%%%%%%%%%%%%%
                                %%%%%%%%%%%%%%%%%%%%%%%%%%%%%%%%%%%%%%%%%%%%%%%%%%%%%%%%%%%%%%%%%%%%%%%%%%%%%%%%%%%%%%%%%%%%%%%%%%%%%%%%%%%%%%%%%%%%%%%%%%%%%%%%%%%%%%%%
%%%%%%%%%%%%%%%%%%%%%%%%%%%%%%%%%%%%%%%%%%%%%%%%%%%%%%%%%%%%%%%%%%%%%%%%%%%%%%%%%%%%%%%%%%%%%%%%%%%%%%%%%%%%%%%%%%%%%%%%%%%%%%%%%%%%%%%%%%%%%%%%%%%%%%%%%%%%%%%%%%%%%%%%%%%%%%%%%%%%%%%%%%%%%%%%%%%%%%%%
%%%%%%%%%%%%%%%%%%%%%%%%%%%%%%%%%%%%%%%%%%%%%%%%%%%%%%%%%%%%%%%%%%%%%%%%%%%%%%%%%%%%%%%%%%%%%%%%%%%%%%%%%%%%%%%%%%%%%%%%%%%%%%%%%%%%%%%%%%%%%%%%%%%%%%%%%%%%%%%%%%%%%%%%%%%%%%%%%%%%%%%%%%%%%%%%%%%%%%%%
\begin{abstract} In this paper, we introduce and investigate the subclass $\mathcal{P}_{p,q}^{\xi ,\kappa}(\tau, \eta)$ of starlike functions with negative coefficients by using the differential operator $\Upsilon_{\tau ,p,q}^{\xi ,\kappa}$. Coefficient inequalities, growth and distortion theorems, closure theorems, and some properties of several functions belonging to this class are obtained. We also determine the radii of close-to-convexity, starlikeness, and convexity for functions belonging to the class $\mathcal{P}_{p,q}^{\xi ,\kappa}(\tau, \eta)$. Furthermore, we obtain the integral means inequality and neighborhood results for functions belonging to the class $\mathcal{P}_{p,q}^{\xi ,\kappa}(\tau, \eta)$. The results presented in this paper generalize or improve those in related works of several earlier authors. \\
\end{abstract}



\section{Introduction and definitions}
Let $\mathcal{A}$ denote the class of all analytic functions $\varphi$ defined in the open unit disk $\Omega=\{z\in\mathbb{C}:\left\vert z\right\vert <1\}$ and normalized by the conditions $\varphi(0)=0$ and $\varphi^{\prime}(0)=1$. Thus
each $\varphi\in\mathcal{A}$ has a Taylor-Maclaurin series expansion of the form:
\begin{equation} \label{ieq1}
\varphi(z)=z+\sum\limits_{\nu=2}^{\infty}c_{\nu}z^{\nu}, \ \  (z \in\Omega).
\end{equation}

Further, let $\mathcal{S}$ denote the class of all functions $\varphi \in\mathcal{A}$ which are univalent in $\Omega$.

The quantum calculus (abbreviated by $q$-calculus) is one of the important tools that is used to investigate subclasses of analytic functions. Kanas and R\u{a}ducanu \cite{I9} have used the fractional $q$-calculus operators in investigations of certain classes of functions which are analytic in $\Omega.$ Recently, the area of $q$-calculus has attracted the serious attention of researchers and several new operators have been proposed. This great interest is due to its applications in various branches of mathematics and physics, as for example, in the areas of ordinary fractional calculus, optimal control problems, quantum physics, operator theory, and q-transform analysis. The application of $q$-calculus was initiated by Jackson \cite{I7}. He was the first to develop $q$-integral and $q$-derivative in a systematic way. Later, geometrical interpretation of $q$-analysis has been recognized through studies on quantum groups. Simply, the $q$-calculus is classical calculus without the notion of limits. A comprehensive study on applications of $q$-calculus in operator theory may be found in \cite{I8}, see also \cite{AD1,AD2}.

The $q$-calculus is based on one parameter, the generalization of $q$-calculus is the post-quantum calculus (abbreviated by $(p, q)$-calculus). The $q$-calculus may be obtained by substituting $p = 1$ in $(p, q)$-calculus. We suppose throughout this section that $0<p<q\leq 1$. For the convenience, some basic definitions and notations of $(p, q)$-calculus are mentioned below:

For $0<p<q\leq 1$ the Jackson's $(\mathit{p,q)}$\textit{-derivative} of a
function $\varphi\in \mathcal{A}$ is, by definition, given as follows \cite{I7}
\begin{equation}
\mathbf{D}_{p,q}\varphi(z)=\left\{
\begin{array}{lcl}
\dfrac{\varphi(pz)-\varphi(qz)}{(p-q)z} & for & z\neq 0, \\
\varphi^{\prime }(0) & for & z=0.%
\end{array}%
\right.  \label{in5}
\end{equation}

From (\ref{in5}), we have%
\begin{equation}
\mathbf{D}_{p,q}\varphi(z)=1+\sum\limits_{\nu=2}^{\infty }[\nu]_{p,q}c_{\nu}z^{\nu-1},  \label{dqf}
\end{equation}%
where%
\begin{equation}
\lbrack \nu]_{p,q}=p^{\nu-1}+p^{\nu-2}q+p^{\nu-3}q^2+\cdots+pq^{\nu-2}+q^{\nu-1}= \frac{p^{\nu}-q^{\nu}}{p-q},  \label{in6}
\end{equation}%
is called $(p, q)$-\textit{bracket} or \textit{twin-basic number}. Note that for $p=1$, the twin-basic number is a natural generalization of the $q$-number, that is
\begin{equation*}
\lbrack \nu]_{1,q}= \frac{1-q^{\nu}}{1-q}= \lbrack \nu]_{q}, \ \ q\neq1.  \label{in6}
\end{equation*}%

Note also that for $p=1$, the Jackson's $(p,q)$-derivative reduces to the
Jackson's $q$-derivative, see \cite{I7}.

Its clearly verified that for a function $\psi(z)=z^{\nu},$ we obtain $\mathbf{D}_{p,q}\psi(z)=\mathbf{D}_{p,q}z^{\nu}=\frac{%
p^{\nu}-q^{\nu}}{p-q}z^{\nu-1}=[\nu]_{p,q}z^{\nu-1}$.

For $\varphi\in \mathcal{A},$ we define the S\u{a}l\u{a}gean $(p,q)$-differential operator as follows:%
\begin{eqnarray}
\Psi_{p,q}^{0}\varphi(z) &=&\varphi(z),  \notag  \\
\Psi_{p,q}^{1}\varphi(z) &=&z\mathbf{D}_{p,q}\varphi(z), \notag \\
& \vdots \notag \\
\Psi_{p,q}^{\kappa}\varphi(z) &=&\Psi_{p,q}^{1}\big(\Psi%
_{p,q}^{\kappa-1}\varphi(z)\big)  \notag \\
 &=&z+\sum\limits_{\nu=2}^{\infty
}[\nu]_{p,q}^{\kappa}c_{\nu}z^{\nu}\quad (\kappa\in \mathbb{N}_{0}:=\mathbb{N\cup \{}0\mathbb{%
\}}; z\in \Omega).
\end{eqnarray}%

We observe that if $p=1$ and $\lim_{q}\rightarrow 1^{-},$ we obtain the familiar S\u{a}l\u{a}gean differential operator \cite{sal}:
\begin{equation}
\Psi^{\kappa}\varphi(z)=z+\sum\limits_{\nu=2}^{\infty }\nu^{\kappa}c_{\nu}z^{\nu}\quad (\kappa\in \mathbb{N}%
_{0}; z\in \Omega ).  \label{gsala}
\end{equation}

Now let%
\begin{eqnarray}
\Upsilon_{\tau ,p,q}^{0,\kappa}\varphi(z)&=&\Psi_{p,q}^{\kappa}\varphi(z), \notag \\
\Upsilon_{\tau ,p,q}^{1,\kappa}\varphi(z) &=&(1-\tau )\Psi%
_{p,q}^{\kappa}\varphi(z)+\tau z\big(\Psi_{p,q}^{\kappa}\varphi(z)\big)^{\prime } \label{rq1} \\
&=&z+\sum\limits_{\nu=2}^{\infty }[\nu]_{p,q}^{\kappa}[1+(\nu-1)\tau ]c_{\nu}z^{\nu}, \notag \\
\Upsilon_{\tau ,p,q}^{2,\kappa}\varphi(z) &=&(1-\tau )\Upsilon_{\tau
,p,q}^{1,\kappa}\varphi(z)+\tau z\big(\Upsilon_{\tau ,p,q}^{1,\kappa}\varphi(z)\big)^{\prime }
\notag \\
&=&z+\sum\limits_{\nu=2}^{\infty }[\nu]_{p,q}^{\kappa}[1+(\nu-1)\tau ]^{2}c_{\nu}z^{\nu}.
\end{eqnarray}%

In general, we have%
\begin{eqnarray*}
\Upsilon_{\tau ,p,q}^{\xi ,\kappa}\varphi(z) &=&(1-\tau )\Upsilon_{\tau
,p,q}^{\xi -1,\kappa}\varphi(z)+\tau z\big(\Upsilon_{\tau ,p,q}^{\xi
-1,\kappa}\varphi(z)\big)^{\prime } \\
&=&z+\sum\limits_{\nu=2}^{\infty }[\nu]_{p,q}^{\kappa}[1+(\nu-1)\tau ]^{\xi
}c_{\nu}z^{\nu}\ \ \ \ (\tau \geq 0;\text{ } \kappa,\xi \in \mathbb{N}_{0}).
\end{eqnarray*}%

Clearly, we have $\Upsilon_{\tau ,p,q}^{0,0}\varphi(z)=\varphi(z)$ and $\Upsilon%
_{1,p,q}^{1,0}\phi(z)=z\varphi^{\prime }(z).$

We observe that when $p=1,$ we get the differential operator $\Upsilon_{\tau,q}^{\xi ,\kappa}\varphi(z)$ defined and studied by Frasin and
Murugusundaramoorthy \cite{fran}. Also, we observe that when $p=1$ and $\lim_{q\rightarrow 1^{-}},$ we get the differential operator:%
\begin{equation*}
\Upsilon_{\tau }^{\xi ,\kappa}\varphi(z)=z+\sum\limits_{\nu=2}^{\infty
}\nu^{\kappa}[1+(\nu-1)\tau ]^{\xi }c_{\nu}z^{\nu}\ \ \ \ (\tau \geq 0;\text{ }%
\kappa,\xi \in \mathbb{N}_{0}).
\end{equation*}%

We note that when $\kappa=0$, we get the differential operator $\Upsilon_{\tau}^{\xi }$
defined by Al-Oboudi \cite{ob}, and if $\xi =0$, we get S\u{a}l\u{a}gean
differential operator $\Upsilon^{\kappa}$ \cite{sal}.

With the aid of the differential operator $\Upsilon_{\tau,p,q}^{\xi ,\kappa},$ we say that a function $\varphi(z)$ belonging to $\mathcal{A}$
is in the class $\mathcal{Q}_{p,q}^{\xi ,\kappa}(\tau, \eta)$ if and only if%
\begin{equation}
{\mbox{Re}}\left\{ \frac{(1-\tau)z\big(\Upsilon_{\tau
,p,q}^{\xi ,\kappa}\varphi(z)\big)^{\prime}+\tau z\big(\Upsilon_{\tau
,p,q}^{\xi ,\kappa+1}\varphi(z)\big)^{\prime}}{(1-\tau)\Upsilon_{\tau ,p,q}^{\xi ,\kappa}\varphi(z)+\tau\Upsilon_{\tau ,p,q}^{\xi ,\kappa+1}\varphi(z)}%
 \right\} >\eta \ \ \ \ (\kappa,\xi \in \mathbb{N}_{0}),  \label{s}
\end{equation}%
for some $\eta(0\leq\eta<1)$ and $\tau(\tau \geq 0),$ and for all $z\in \Omega$.

Let $\mathcal{T}$ denote the subclass of $\mathcal{A}$ consisting functions of the form:%
\begin{equation}
\varphi(z)=z-\sum\limits_{\nu=2}^{\infty }c_{\nu}z^{\nu} \ \ \ \ (c_{\nu}\geq 0; z\in \Omega).
\label{ieq1n}
\end{equation}%

Further, we define the class $\mathcal{P}_{p,q}^{\xi ,\kappa}(\tau, \eta)$ by%
\begin{equation}
\mathcal{P}_{p,q}^{\xi ,\kappa}(\tau, \eta):=\mathcal{Q}_{p,q}^{\xi ,\kappa}(\tau, \eta)\cap \mathcal{T}.
\end{equation}%

The present paper aims at providing a systematic investigation of the various interesting properties and characteristics of the class $\mathcal{P}_{p,q}^{\xi ,\kappa}(\tau, \eta)$. Some interesting corollaries and consequences of the main results are also considered. We employ techniques similar to these used earlier by Al-Hawary et al. \cite{A1,A4}, Amourah et al. \cite{A2,C4}, Aouf and Srivastava \cite{A5}, and Frasin et al. \cite{A3}.
%*************************************************************************

\section{Coefficient estimates}
We begin this section by obtaining a necessary and sufficient condition for a function $\varphi(z)$ to be in the class $\mathcal{P}_{p,q}^{\xi ,\kappa}(\tau, \eta)$.
\begin{theorem}
\label{thm21} Let the function $\varphi(z)$ be defined by (\ref{ieq1n}). Then $\varphi(z) \in \mathcal{P}_{p,q}^{\xi ,\kappa}(\tau, \eta)$ if and only if
\begin{equation} \label{theq21}
\sum\limits_{\nu=2}^{\infty }[\nu]_{p,q}^{\kappa}(\nu-\eta)\big\{1+([\nu]_{p,q}-1)\tau\big\}[1+(\nu-1)\tau ]^{\xi} c_{\nu}\leq 1-\eta.
\end{equation}

The result is sharp.
\end{theorem}

\begin{proof}
Assume that the function $\varphi(z)$ is in the class $\mathcal{P}_{p,q}^{\xi ,\kappa}(\tau, \eta)$. Then we have
\begin{equation}
{\mbox{Re}}\left\{ \frac{(1-\tau)z\big(\Upsilon_{\tau
,p,q}^{\xi ,\kappa}\varphi(z)\big)^{\prime}+\tau z\big(\Upsilon_{\tau
,p,q}^{\xi ,\kappa+1}\varphi(z)\big)^{\prime}}{(1-\tau)\Upsilon_{\tau ,p,q}^{\xi ,\kappa}\varphi(z)+\tau\Upsilon_{\tau ,p,q}^{\xi ,\kappa+1}\varphi(z)}%
 \right\} = {\mbox{Re}}\left\{ \frac{K}{L} \right\} >\eta\ \ \ \ (\kappa,\xi \in \mathbb{N}_{0}),
\end{equation}%
for some $\eta(0\leq\eta<1)$ and $\tau(\tau \geq 0),$ and for all $z\in \Omega$. Now
\begin{eqnarray}
\hspace{-.5in} K &=&(1-\tau )\left(z-\sum\limits_{\nu=2}^{\infty }\nu[\nu]_{p,q}^{\kappa}[1+(\nu-1)\tau ]^{\xi
}c_{\nu}z^{\nu}\right)+\tau \left(z-\sum\limits_{\nu=2}^{\infty}\nu[\nu]_{p,q}^{\kappa+1}[1+(\nu-1)\tau ]^{\xi}c_{\nu}z^{\nu}\right)  \notag \\
&=&z-\sum\limits_{\nu=2}^{\infty }\nu[\nu]_{p,q}^{\kappa}\big\{1+([\nu]_{p,q}-1)\tau\big\}[1+(\nu-1)\tau ]^{\xi} c_{\nu}z^{\nu},
\end{eqnarray}%
and
\begin{eqnarray}
\hspace{-.5in} L &=&(1-\tau )\left(z-\sum\limits_{\nu=2}^{\infty }[\nu]_{p,q}^{\kappa}[1+(\nu-1)\tau ]^{\xi
}c_{\nu}z^{\nu}\right)+\tau \left(z-\sum\limits_{\nu=2}^{\infty}[\nu]_{p,q}^{\kappa+1}[1+(\nu-1)\tau ]^{\xi}c_{\nu}z^{\nu}\right)  \notag  \\
&=&z-\sum\limits_{\nu=2}^{\infty }[\nu]_{p,q}^{\kappa}\big\{1+([\nu]_{p,q}-1)\tau\big\}[1+(\nu-1)\tau ]^{\xi} c_{\nu}z^{\nu}.
\end{eqnarray}%

Consequently,
\begin{eqnarray}
{\mbox{Re}}\left\{ \frac{(1-\tau)z\big(\Upsilon_{\tau
,p,q}^{\xi ,\kappa}\varphi(z)\big)^{\prime}+\tau z\big(\Upsilon_{\tau
,p,q}^{\xi ,\kappa+1}\varphi(z)\big)^{\prime}}{(1-\tau)\Upsilon_{\tau ,p,q}^{\xi ,\kappa}\varphi(z)+\tau\Upsilon_{\tau ,p,q}^{\xi ,\kappa+1}\varphi(z)}%
 \right\} \hspace{2in}  \\ \nonumber
 = {\mbox{Re}}\left\{ \frac{z-\sum\limits_{\nu=2}^{\infty }\nu[\nu]_{p,q}^{\kappa}\big\{1+([\nu]_{p,q}-1)\tau\big\}[1+(\nu-1)\tau ]^{\xi} c_{\nu}z^{\nu}}{z-\sum\limits_{\nu=2}^{\infty }[\nu]_{p,q}^{\kappa}\big\{1+([\nu]_{p,q}-1)\tau\big\}[1+(\nu-1)\tau ]^{\xi} c_{\nu}z^{\nu}} \right\} >\eta.
\end{eqnarray}%

Letting $z\rightarrow 1^{-}$ along the real axis, we can see that
\begin{eqnarray}
1-\sum\limits_{\nu=2}^{\infty }\nu[\nu]_{p,q}^{\kappa}\big\{1+([\nu]_{p,q}-1)\tau\big\}[1+(\nu-1)\tau ]^{\xi} c_{\nu} \hspace{2in}  \\ \nonumber
\geq \eta \left(1-\sum\limits_{\nu=2}^{\infty }[\nu]_{p,q}^{\kappa}\big\{1+([\nu]_{p,q}-1)\tau\big\}[1+(\nu-1)\tau ]^{\xi} c_{\nu}\right).
\end{eqnarray}%

Thus we have the inequality (\ref{theq21}).

Conversely, assume that the inequality (\ref{theq21}) holds true. Then we find that
\begin{eqnarray*}
&& \hspace{-1.2in} \left\vert \frac{(1-\tau)z\big(\Upsilon_{\tau
,p,q}^{\xi ,\kappa}\varphi(z)\big)^{\prime}+\tau z\big(\Upsilon_{\tau
,p,q}^{\xi ,\kappa+1}\varphi(z)\big)^{\prime}}{(1-\tau)\Upsilon_{\tau ,p,q}^{\xi ,\kappa}\varphi(z)+\tau\Upsilon_{\tau ,p,q}^{\xi ,\kappa+1}\varphi(z)} - 1 \right\vert \\
&\leq& \frac{\sum\limits_{\nu=2}^{\infty }[\nu]_{p,q}^{\kappa} (\nu-1)\big\{1+([\nu]_{p,q}-1)\tau\big\}[1+(\nu-1)\tau ]^{\xi} c_{\nu}|z|^{\nu-1}}{1-\sum\limits_{\nu=2}^{\infty }[\nu]_{p,q}^{\kappa}\big\{1+([\nu]_{p,q}-1)\tau\big\}[1+(\nu-1)\tau ]^{\xi} c_{\nu}|z|^{\nu-1}} \\
&\leq& \frac{\sum\limits_{\nu=2}^{\infty }[\nu]_{p,q}^{\kappa} (\nu-1)\big\{1+([\nu]_{p,q}-1)\tau\big\}[1+(\nu-1)\tau ]^{\xi} c_{\nu}}{1-\sum\limits_{\nu=2}^{\infty }[\nu]_{p,q}^{\kappa}\big\{1+([\nu]_{p,q}-1)\tau\big\}[1+(\nu-1)\tau ]^{\xi} c_{\nu}} \\
&\leq& 1- \eta.
\end{eqnarray*}%

This shows that the values of the function
\begin{eqnarray}
\varphi (z)= \frac{(1-\tau)z\big(\Upsilon_{\tau
,p,q}^{\xi ,\kappa}\varphi(z)\big)^{\prime}+\tau z\big(\Upsilon_{\tau
,p,q}^{\xi ,\kappa+1}\varphi(z)\big)^{\prime}}{(1-\tau)\Upsilon_{\tau ,p,q}^{\xi ,\nu}\varphi(z)+\tau\Upsilon_{\tau ,p,q}^{\xi ,\nu+1}\varphi(z)}
\end{eqnarray}%
lie in a circle which is centered at $w=1$ and whose radius is $1-\eta$. Hence $\varphi(z)$ satisfies the condition (\ref{s}).

Finally, the function $\varphi(z)$ given by
\begin{eqnarray}
\varphi(z)=z-\frac{1-\eta }{[\nu]_{p,q}^{\kappa}(\nu-\eta)\big\{1+([\nu]_{p,q}-1)\tau\big\}[1+(\nu-1)\tau ]^{\xi}}z^{\nu} \ \ \ \ (\nu \geq 2) \label{fun1}
\end{eqnarray}%
is an extremal function for the assertion of Theorem \ref{thm21}. This completes the proof of Theorem \ref{thm21}.
\end{proof}

\begin{corollary} % \cite{C28}
Let the function $\varphi(z)$, defined by (\ref{ieq1n}), be in the class $\mathcal{P}_{p,q}^{\xi ,\kappa}(\tau, \eta)$. Then
\begin{eqnarray}
c_{\nu} \leq \frac{1-\eta }{[\nu]_{p,q}^{\kappa}(\nu-\eta)\big\{1+([\nu]_{p,q}-1)\tau\big\}[1+(\nu-1)\tau ]^{\xi}} \quad \ \ \ (\nu \geq 2). \label{co1}
\end{eqnarray}%

The equality in (\ref{co1}) is achieved for the function $\varphi(z)$ given by (\ref{fun1}).
\end{corollary}
%*************************************************************************

\section{Inclusion relations}
We begin this section by showing the following inclusion relation.
\begin{theorem}
\label{thm22} Let \ $0\leq\eta_1\leq\eta_2<1$, $0\leq\tau\leq1$, and $\kappa,\xi \in \mathbb{N}_{0}$. Then
\begin{equation} \label{theq22}
\mathcal{P}_{p,q}^{\xi ,\kappa}(\tau, \eta_1)\supseteq\mathcal{P}_{p,q}^{\xi ,\kappa}(\tau, \eta_2).
\end{equation}
\end{theorem}

\begin{proof}
Let the function $\varphi(z)$ defined by (\ref{ieq1n}) be in the class $\mathcal{P}_{p,q}^{\xi ,\kappa}(\tau, \eta_2)$ and let $\eta_1=\eta_2-\delta$. Then, by Theorem \ref{thm21}, we have
\begin{equation}
\sum\limits_{\nu=2}^{\infty }[\nu]_{p,q}^{\kappa}(\nu-\eta_2)\big\{1+([\nu]_{p,q}-1)\tau\big\}[1+(\nu-1)\tau ]^{\xi} c_{\nu}\leq 1-\eta_2
\end{equation}
and
\begin{equation}
\sum\limits_{\nu=2}^{\infty }[\nu]_{p,q}^{\kappa}\big\{1+([\nu]_{p,q}-1)\tau\big\}[1+(\nu-1)\tau ]^{\xi} c_{\nu}\leq \frac{1-\eta_2}{2-\eta_2}<1.
\end{equation}

Consequently,
\begin{eqnarray*}
\sum\limits_{\nu=2}^{\infty }[\nu]_{p,q}^{\kappa}(\nu-\eta_1)\big\{1&+&([\nu]_{p,q}-1)\tau\big\}[1+(\nu-1)\tau ]^{\xi} c_{\nu}  \\
&=& \sum\limits_{\nu=2}^{\infty }[\nu]_{p,q}^{\kappa}(\nu-\eta_2)\big\{1+([\nu]_{p,q}-1)\tau\big\}[1+(\nu-1)\tau ]^{\xi} c_{\nu}\\
&+& \delta\sum\limits_{\nu=2}^{\infty }[\nu]_{p,q}^{\kappa}\big\{1+([\nu]_{p,q}-1)\tau\big\}[1+(\nu-1)\tau ]^{\xi} c_{\nu} \\
&\leq& 1-\eta_1.
\end{eqnarray*}%

This completes the proof of Theorem \ref{thm22}.
\end{proof}

\begin{theorem}
\label{thm23} Let \ $0\leq\eta<1$, $0\leq\tau_1\leq\tau_2\leq1$, and $\kappa,\xi \in \mathbb{N}_{0}$. Then
\begin{equation} \label{theq23}
\mathcal{P}_{p,q}^{\xi ,\kappa}(\tau_1, \eta)\supseteq\mathcal{P}_{p,q}^{\xi ,\kappa}(\tau_2, \eta).
\end{equation}
\end{theorem}

\begin{proof}
Let the function $\varphi(z)$ defined by (\ref{ieq1n}) be in the class $\mathcal{P}_{p,q}^{\xi ,\kappa}(\tau_2, \eta)$. Then, by Theorem \ref{thm21}, we have
\begin{eqnarray*}
\sum\limits_{\nu=2}^{\infty }[\nu]_{p,q}^{\kappa}(\nu-\eta)\big\{1&+&([\nu]_{p,q}-1)\tau_1\big\}[1+(\nu-1)\tau_1 ]^{\xi} c_{\nu}  \\
&\leq& \sum\limits_{\nu=2}^{\infty }[\nu]_{p,q}^{\kappa}(\nu-\eta)\big\{1+([\nu]_{p,q}-1)\tau_2\big\}[1+(\nu-1)\tau_2 ]^{\xi} c_{\nu}\\
&\leq& 1-\eta.
\end{eqnarray*}%

This completes the proof of Theorem \ref{thm23}.
\end{proof}

\begin{theorem}
\label{thm24} Let \ $0\leq\eta<1$, $0\leq\tau\leq1$, and $\kappa,\xi \in \mathbb{N}_{0}$. Then
\begin{equation} \label{theq24}
\mathcal{P}_{p,q}^{\xi ,\kappa}(\tau, \eta)\supseteq\mathcal{P}_{p,q}^{\xi ,\kappa+1}(\tau, \eta).
\end{equation}
and
\begin{equation}
\mathcal{P}_{p,q}^{\xi ,\kappa}(\tau, \eta)\supseteq\mathcal{P}_{p,q}^{\xi+1 ,\kappa}(\tau, \eta).
\end{equation}
\end{theorem}
The proof of Theorem \ref{thm24} follows also from Theorem \ref{thm21}.
%*************************************************************************
\newpage
\section{Growth and distortion theorems}
\begin{theorem}
\label{thm25} Let the function $\varphi(z)$, defined by (\ref{ieq1n}), be in the class $\mathcal{P}_{p,q}^{\xi ,\kappa}(\tau, \eta)$. Then, for $|z|=r<1$,
\begin{equation} \label{theq25}
|\Upsilon_{\tau ,p,q}^{i ,j}\varphi(z)|\geq r - \frac{1-\eta}{[2]_{p,q}^{\kappa-j}(2-\eta)\big\{1+([2]_{p,q}-1)\tau\big\}(1+\tau)^{\xi-i}} \ r^2
\end{equation}
and
\begin{eqnarray} \label{theq26}
\qquad \ |\Upsilon_{\tau ,p,q}^{i ,j}\varphi(z)|\leq r + \frac{1-\eta}{[2]_{p,q}^{\kappa-j}(2-\eta)\big\{1+([2]_{p,q}-1)\tau\big\}(1+\tau)^{\xi-i}} \ r^2  \\
\qquad \qquad \qquad \qquad \qquad \qquad (0\leq i \leq \xi; 0\leq j \leq \kappa; z\in \Omega). \nonumber
\end{eqnarray}

The equalities in (\ref{theq25}) and (\ref{theq26}) are achieved for the function $\varphi(z)$ given by
\begin{equation} \label{theq27}
\varphi(z)= z- \frac{1-\eta}{[2]_{p,q}^{\kappa}(2-\eta)\big\{1+([2]_{p,q}-1)\tau\big\}(1+\tau)^{\xi}} \ z^2  \ \ \ \ (z=\pm r).
\end{equation}
\end{theorem}

\begin{proof}
Note that the function $\varphi(z)\in\mathcal{P}_{p,q}^{\xi ,\kappa}(\tau, \eta)$ if and only if
\begin{equation*}
\Upsilon_{\tau ,p,q}^{i ,j}\varphi(z)\in\mathcal{P}_{p,q}^{\xi-i ,\kappa-j}(\tau, \eta)
\end{equation*}
and that
\begin{eqnarray} \label{tm251}
\Upsilon_{\tau ,p,q}^{i ,j}\varphi(z)= z-\sum\limits_{\nu=2}^{\infty }[\nu]_{p,q}^{j}[1+(\nu-1)\tau ]^{i
}c_{\nu}z^{\nu}.
\end{eqnarray}%

By Theorem \ref{thm21}, we know that
\begin{eqnarray}
[2]_{p,q}^{\kappa-j}(2-\eta)\big\{1&+&([2]_{p,q}-1)\tau\big\}(1+\tau)^{\xi-i} \sum\limits_{\nu=2}^{\infty }[\nu]_{p,q}^{j} (1+\tau)^{i}c_{\nu} \\ \nonumber
&\leq& \sum\limits_{\nu=2}^{\infty }[\nu]_{p,q}^{\kappa}(\nu-\eta)\big\{1+([\nu]_{p,q}-1)\tau\big\}[1+(\nu-1)\tau ]^{\xi} c_{\nu} \\ \nonumber
&\leq& 1-\eta,
\end{eqnarray}%
which implies,
\begin{eqnarray} \label{tm252}
\sum\limits_{\nu=2}^{\infty }[\nu]_{p,q}^{j} (1+\tau)^{i} c_{\nu}\leq \frac{1-\eta}{[2]_{p,q}^{\kappa-j}(2-\eta)\big\{1+([2]_{p,q}-1)\tau\big\}(1+\tau)^{\xi-i}}.
\end{eqnarray}%

The assertions (\ref{theq25}) and (\ref{theq26}) of Theorem \ref{thm25} would now follow readily from (\ref{tm251}) and (\ref{tm252}).

Finally, we note that the equalities (\ref{theq25}) and (\ref{theq26}) are achieved for the function $\varphi(z)$ defined by
\begin{equation}
\Upsilon_{\tau ,p,q}^{i ,j}\varphi(z)= z- \frac{1-\eta}{[2]_{p,q}^{\kappa-j}(2-\eta)\big\{1+([2]_{p,q}-1)\tau\big\}(1+\tau)^{\xi-i}} \ z^2.
\end{equation}

This completes the proof of Theorem \ref{thm25}.
\end{proof}

Taking $i=j=0$ in Theorem \ref{thm25}, we immediately get the following corollary.

\begin{corollary}
Let the function $\varphi(z)$, defined by (\ref{ieq1n}), be in the class $\mathcal{P}_{p,q}^{\xi ,\kappa}(\tau, \eta)$. Then, for $|z|=r<1$,
\begin{equation} \label{theq2511}
|\varphi(z)|\geq r - \frac{1-\eta}{[2]_{p,q}^{\kappa}(2-\eta)\big\{1+([2]_{p,q}-1)\tau\big\}(1+\tau)^{\xi}} \ r^2
\end{equation}
and
\begin{eqnarray} \label{theq2612}
\qquad \qquad \qquad \quad \ |\varphi(z)|\leq r + \frac{1-\eta}{[2]_{p,q}^{\kappa}(2-\eta)\big\{1+([2]_{p,q}-1)\tau\big\}(1+\tau)^{\xi}} \ r^2  \qquad \ \ (z\in \Omega).
\end{eqnarray}

The equalities in (\ref{theq2511}) and (\ref{theq2612}) are achieved for the function $\varphi(z)$ given by (\ref{theq27}).
\end{corollary}

Setting $i=\tau=1$ and $j=0$ in Theorem \ref{thm25}, and making use of the definition (\ref{rq1}), we get the following corollary.

\begin{corollary}
Let the function $\varphi(z)$, defined by (\ref{ieq1n}), be in the class $\mathcal{P}_{p,q}^{\xi ,\kappa}(\tau, \eta)$. Then, for $|z|=r<1$,
\begin{equation} \label{theq25111}
|\varphi^{\prime}(z)|\geq 1 - \frac{1-\eta}{[2]_{p,q}^{\kappa+1}(2-\eta)(2)^{\xi-1}} \ r
\end{equation}
and
\begin{eqnarray} \label{theq26122}
\qquad \qquad \qquad \quad \ |\varphi^{\prime}(z)|\leq 1 + \frac{1-\eta}{[2]_{p,q}^{\kappa+1}(2-\eta)(2)^{\xi-1}} \ r  \qquad \ \ (z\in \Omega).
\end{eqnarray}

The equalities in (\ref{theq25111}) and (\ref{theq26122}) are achieved for the function $\varphi(z)$ given by
\begin{equation}
\varphi(z)= z- \frac{1-\eta}{[2]_{p,q}^{\kappa+1}(2-\eta)(2)^{\xi}} \ z^2  \ \ \ \ (z=\pm r).
\end{equation}
\end{corollary}
%*************************************************************************

\section{Closure theorems}
In this section, we shall prove that the class $\mathcal{P}_{p,q}^{\xi ,\kappa}(\tau, \eta)$ is closed under convex linear combinations.
\begin{theorem}
\label{thm26} The class $\mathcal{P}_{p,q}^{\xi ,\kappa}(\tau, \eta)$ is a convex set.
\end{theorem}

\begin{proof}
Let the functions
\begin{equation}
\varphi_\alpha(z)=z-\sum\limits_{\nu=2}^{\infty }c_{\alpha,\nu}z^{\nu} \ \ \ \ (c_{\alpha,\nu}\geq 0; \alpha = 1,2; z\in \Omega)
\end{equation}%
be in the class $\mathcal{P}_{p,q}^{\xi ,\kappa}(\tau, \eta)$. It is sufficient to show that the function $\psi(z)$ defined by
\begin{eqnarray}
\psi(z):= \mu \varphi_1(z)+(1-\mu)\varphi_2(z) \ \ \ \ (0\leq \mu \leq1)
\end{eqnarray}%
is also in the class $\mathcal{P}_{p,q}^{\xi ,\kappa}(\tau, \eta)$. Since, for $0\leq \mu \leq1$,
\begin{equation}
\psi(z)=z-\sum\limits_{\nu=2}^{\infty }\big\{\mu c_{1,\nu}+(1-\mu) c_{2,\nu}\big\}z^{c},
\end{equation}%
with the aid of Theorem \ref{thm21}, we have
\begin{eqnarray}
\sum\limits_{\nu=2}^{\infty }[\nu]_{p,q}^{\kappa}(\nu-\eta)\big\{1+([\nu]_{p,q}-1)\tau\big\}[1+(\nu-1)\tau ]^{\xi} \big\{\mu c_{1,\nu}+(1-\mu) c_{2,\nu}\big\} \leq 1-\eta,
\end{eqnarray}%
which implies that $\psi(z)\in\mathcal{P}_{p,q}^{\xi ,\kappa}(\tau, \eta)$. Hence $\mathcal{P}_{p,q}^{\xi ,\kappa}(\tau, \eta)$ is a convex set.
\end{proof}
%\newpage
\begin{theorem}
\label{thm27} Let $\varphi_1(z)=z$ and
\begin{eqnarray}
\hspace{-0.2in} \varphi_\nu(z)=z-\frac{1-\eta }{[\nu]_{p,q}^{\kappa}(\nu-\eta)\big\{1+([\nu]_{p,q}-1)\tau\big\}[1+(\nu-1)\tau ]^{\xi}}z^{\nu} \ \ (\nu \geq 2; \kappa,\xi \in \mathbb{N}_{0})
\end{eqnarray}%
for $0\leq \eta <1$ and $0\leq \tau \leq1$. Then the function $\varphi(z)$ is in the class $\mathcal{P}_{p,q}^{\xi ,\kappa}(\tau, \eta)$ if and only if it can be expressed in the form:
\begin{equation}
\varphi(z)=\sum\limits_{\nu=1}^{\infty }\mu_{\nu}\varphi_\nu(z), \label{fun111}
\end{equation}%
where
\begin{equation}
\mu_{\nu} \geq 0 \ (\nu\geq1) \text{ \ and \ } \sum\limits_{\nu=1}^{\infty }\mu_{\nu}=1.
\end{equation}%
\end{theorem}

\begin{proof}
Assume that
\begin{eqnarray}
\varphi(z)&=&\sum\limits_{\nu=1}^{\infty }\mu_{\nu}\varphi_\nu(z) \\ \nonumber
&=&z-\sum\limits_{\nu=2}^{\infty }\frac{1-\eta }{[\nu]_{p,q}^{\kappa}(\nu-\eta)\big\{1+([\nu]_{p,q}-1)\tau\big\}[1+(\nu-1)\tau ]^{\xi}} \ \mu_{\nu}z^{\nu}.
\end{eqnarray}%

Then it follows that
\begin{eqnarray*}
\sum\limits_{\nu=2}^{\infty }\frac{[\nu]_{p,q}^{\kappa}(\nu-\eta)\big\{1+([\nu]_{p,q}-1)\tau\big\}[1+(\nu-1)\tau ]^{\xi}}{1-\eta}\cdot \hspace{2in} \\
\frac{1-\eta }{[\nu]_{p,q}^{\kappa}(\nu-\eta)\big\{1+([\nu]_{p,q}-1)\tau\big\}[1+(\nu-1)\tau ]^{\xi}} \ \mu_{\nu} = \sum\limits_{\nu=2}^{\infty }\mu_{\nu}=1-\mu_{1} \leq1.
\end{eqnarray*}%

Thus, by Theorem \ref{thm21}, $\varphi(z)\in\mathcal{P}_{p,q}^{\xi ,\kappa}(\tau, \eta)$.

Conversely, assume that the function $\varphi(z)$ defined by (\ref{ieq1n}) belongs to the class $\mathcal{P}_{p,q}^{\xi ,\kappa}(\tau, \eta)$. Then
\begin{eqnarray*}
c_\nu\leq\frac{1-\eta }{[\nu]_{p,q}^{\kappa}(\nu-\eta)\big\{1+([\nu]_{p,q}-1)\tau\big\}[1+(\nu-1)\tau ]^{\xi}} \ \ \ (\nu \geq 2; \kappa,\xi \in \mathbb{N}_{0}).
\end{eqnarray*}%

Setting
\begin{eqnarray*}
\mu_\nu=\frac{[\nu]_{p,q}^{\kappa}(\nu-\eta)\big\{1+([\nu]_{p,q}-1)\tau\big\}[1+(\nu-1)\tau ]^{\xi}}{1-\eta} \ c_\nu \ \ \ (\nu \geq 2; \kappa,\xi \in \mathbb{N}_{0})
\end{eqnarray*}%
and
\begin{eqnarray*}
\mu_1=1-\sum\limits_{\nu=2}^{\infty }\mu_{\nu},
\end{eqnarray*}%
we can see that $\varphi(z)$ can be expressed in the form (\ref{fun111}). This completes the proof of Theorem \ref{thm27}.
\end{proof}
%*************************************************************************

\section{Radii of close-to-convexity, starlikenss, and convexity}
In this section, we shall determine the radii of close-to-convexity, starlikeness, and convexity for functions belonging to the class $\mathcal{P}_{p,q}^{\xi ,\kappa}(\tau, \eta)$.
\begin{theorem}
\label{thm28} Let the function $\varphi(z)$, defined by (\ref{ieq1n}), be in the class $\mathcal{P}_{p,q}^{\xi ,\kappa}(\tau, \eta)$. Then $\varphi(z)$ is close-to-convex of order $\rho (0\leq \rho <1)$ in $|z|< r_1$, where
\begin{equation}
r_{1}:=\inf_{\nu}\left( \frac{(1-\rho)\nu^{-1}[\nu]_{p,q}^{\kappa}(\nu-\eta)\big\{1+([\nu]_{p,q}-1)\tau\big\}[1+(\nu-1)\tau ]^{\xi}}{1-\eta}\right) ^{1 / (\nu-1)} \ \ \ \ (\nu\geq2). \label{fun112}
\end{equation}%

The result is sharp, with the extremal function $\varphi(z)$ given by (\ref{fun1}).
\end{theorem}

\begin{proof}
We need to show that
\begin{equation*}
|\varphi^{\prime}(z)-1| \leq 1-\rho \ \text{ \ for \ } |z|<r_1,
\end{equation*}%
where $r_1$ is given by (\ref{fun112}). Indeed, definition (\ref{ieq1n}) implies that
\begin{equation*}
|\varphi^{\prime}(z)-1| \leq \sum\limits_{\nu=2}^{\infty } \nu c_\nu |z|^{\nu-1}.
\end{equation*}%

Thus,
\begin{equation*}
|\varphi^{\prime}(z)-1| \leq 1-\rho,
\end{equation*}%
if
\begin{equation}
\sum\limits_{\nu=2}^{\infty } \left( \frac{\nu}{1-\rho} \right) c_\nu |z|^{\nu-1} \leq 1. \label{fun123}
\end{equation}%

But, by Theorem \ref{thm21}, (\ref{fun123}) holds true if
\begin{equation}
\left( \frac{\nu}{1-\rho} \right) |z|^{\nu-1} \leq \frac{[\nu]_{p,q}^{\kappa}(\nu-\eta)\big\{1+([\nu]_{p,q}-1)\tau\big\}[1+(\nu-1)\tau ]^{\xi}}{1-\eta},
\end{equation}%
that is, if
\begin{equation}
|z|\leq\left( \frac{(1-\rho)\nu^{-1}[\nu]_{p,q}^{\kappa}(\nu-\eta)\big\{1+([\nu]_{p,q}-1)\tau\big\}[1+(\nu-1)\tau ]^{\xi}}{1-\eta}\right) ^{1 / (\nu-1)} \ \ \ \ (\nu\geq2). \label{fun321}
\end{equation}%

Theorem \ref{thm28} follows readily from (\ref{fun321}).
\end{proof}

\begin{theorem}
\label{thm29} Let the function $\varphi(z)$, defined by (\ref{ieq1n}), be in the class $\mathcal{P}_{p,q}^{\xi ,\kappa}(\tau, \eta)$. Then $\varphi(z)$ is starlike of order $\rho (0\leq \rho <1)$ in $|z|< r_2$, where
\begin{equation}
r_{2}:=\inf_{\nu}\left( \frac{(1-\rho)[\nu]_{p,q}^{\kappa}(\nu-\eta)\big\{1+([\nu]_{p,q}-1)\tau\big\}[1+(\nu-1)\tau ]^{\xi}}{(\nu-\rho)(1-\eta)}\right) ^{1 / (\nu-1)} \ \ \ \ (\nu\geq2). \label{fun222}
\end{equation}%

The result is sharp, with the extremal function $\varphi(z)$ given by (\ref{fun1}).
\end{theorem}

\begin{proof}
We need to show that
\begin{equation*}
\left\vert\frac{z\varphi^{\prime}(z)}{\varphi(z)}-1\right\vert \leq 1-\rho \ \text{ \ for \ } |z|<r_2,
\end{equation*}%
where $r_2$ is given by (\ref{fun222}). Indeed, definition (\ref{ieq1n}) implies that
\begin{equation*}
\left\vert\frac{z\varphi^{\prime}(z)}{\varphi(z)}-1\right\vert \leq \frac{\sum\limits_{\nu=2}^{\infty } (\nu-1) c_\nu |z|^{\nu-1}}{1-\sum\limits_{\nu=2}^{\infty } c_\nu |z|^{\nu-1}}.
\end{equation*}%

Thus,
\begin{equation*}
\left\vert\frac{z\varphi^{\prime}(z)}{\varphi(z)}-1\right\vert \leq 1-\rho,
\end{equation*}%
if
\begin{equation}
\sum\limits_{\nu=2}^{\infty } \left( \frac{\nu-\rho}{1-\rho} \right) c_\nu |z|^{\nu-1} \leq 1. \label{fun444}
\end{equation}%

But, by Theorem \ref{thm21}, (\ref{fun444}) holds true if
\begin{equation}
\left( \frac{\nu-\rho}{1-\rho} \right) |z|^{\nu-1} \leq \frac{[\nu]_{p,q}^{\kappa}(\nu-\eta)\big\{1+([\nu]_{p,q}-1)\tau\big\}[1+(\nu-1)\tau ]^{\xi}}{1-\eta},
\end{equation}%
that is, if
\begin{equation}
|z|\leq\left( \frac{(1-\rho)[\nu]_{p,q}^{\kappa}(\nu-\eta)\big\{1+([\nu]_{p,q}-1)\tau\big\}[1+(\nu-1)\tau ]^{\xi}}{(\nu-\rho)(1-\eta)}\right) ^{1 / (\nu-1)} \ \ \ \ (\nu\geq2). \label{fun555}
\end{equation}%

Theorem \ref{thm29} follows readily from (\ref{fun555}).
\end{proof}

\begin{corollary}
Let the function $\varphi(z)$, defined by (\ref{ieq1n}), be in the class $\mathcal{P}_{p,q}^{\xi ,\kappa}(\tau, \eta)$. Then $\varphi(z)$ is convex of order $\rho (0\leq \rho <1)$ in $|z|< r_3$, where
\begin{equation}
r_{3}:=\inf_{\nu}\left( \frac{(1-\rho)\nu^{-1}[\nu]_{p,q}^{\kappa}(\nu-\eta)\big\{1+([\nu]_{p,q}-1)\tau\big\}[1+(\nu-1)\tau ]^{\xi}}{(\nu-\rho)(1-\eta)}\right) ^{1 / (\nu-1)} \ \ \ \ (\nu\geq2).
\end{equation}%

The result is sharp, with the extremal function $\varphi(z)$ given by (\ref{fun1}).
\end{corollary}

\section{Integral means inequality}
For any two functions $\varphi$ and $\psi$ analytic in $\Omega$, $\varphi$ is said to be subordinate to $\psi$ in $\Omega$, written $\varphi(z)$ $\prec$ $\psi(z)$, if there exists a Schwarz function $\omega(z)$, analytic in $\Omega$, with
\begin{center}
$\omega(0) = 0$ and $\left\vert \omega(z)\right\vert <1$ for all $z \in\Omega$,
\end{center}
such that $\varphi(z)=\psi\left(\omega(z)\right)$ for all $z \in\Omega$. Furthermore, if the function $\psi$ is univalent in $\Omega$, then we have the following equivalence \cite{C2}:
\[
\varphi(z)\prec \psi(z)\Leftrightarrow \varphi(0)=\psi(0)\text{ and }\varphi(\Omega)\subset
\psi(\Omega).
\]

In order to prove the integral means inequality for functions belonging to the class $\mathcal{P}_{p,q}^{\xi ,\kappa}(\tau, \eta)$, we need the following subordination result due to Littlewood \cite{wood}.
\begin{lemma} \label{l1}
If the functions $\varphi$ and $\psi$ are analytic in $\Omega$ with $\varphi(z)\prec \psi(z)$, then for $\gamma>0$ and $z=re^{i\theta}(0<r<1)$,
\begin{eqnarray} \label{leq1}
\int\limits_{0}^{2\pi}\left\vert \varphi(z)\right\vert ^{\gamma}d\theta \leq \int\limits_{0}^{2\pi }\left\vert \psi(z)\right\vert ^{\gamma}d\theta.
\end{eqnarray}
\end{lemma}

Applying Theorem \ref{thm21} with the extremal function and Lemma \ref{l1}, we achieve the following theorem.
\begin{theorem} \label{im}
Let $\Big\{ [\nu]_{p,q}^{\kappa}(\nu-\eta)\big\{1+([\nu]_{p,q}-1)\tau\big\}[1+(\nu-1)\tau ]^{\xi} \Big\} _{\nu=2}^{\infty }$ be a nondecreasing sequence. If $\varphi \in\mathcal{P}_{p,q}^{\xi ,\kappa}(\tau, \eta)$, then
\begin{equation} \label{teq1}
\int\limits_{0}^{2\pi }\left\vert \varphi(re^{i\theta })\right\vert ^{\gamma}d\theta \leq \int\limits_{0}^{2\pi }\left\vert \varphi_{*}(re^{i\theta})\right\vert ^{\gamma}d\theta \ \ \ \ (0<r<1;\gamma>0),
\end{equation}
where
\begin{equation} \label{teq2}
\varphi_{*}(z)=z - \frac{1-\eta }{[2]_{p,q}^{\kappa}(2-\eta)\big\{1+([2]_{p,q}-1)\tau\big\}(1+\tau)^{\xi}} \ z^{2}.
\end{equation}

\end{theorem}

\begin{proof}
Let the function $\varphi(z)$, defined by (\ref{ieq1n}), be in the class $\mathcal{P}_{p,q}^{\xi ,\kappa}(\tau, \eta)$. Then we need to show that
\begin{equation} \label{eq1}
\int\limits_{0}^{2\pi}\left\vert 1-\sum\limits_{\nu=2}^{\infty}c_{\nu}z^{\nu-1}\right\vert ^{\gamma}d\theta \leq \int\limits_{0}^{2\pi}\left\vert 1-\frac{1-\eta }{[2]_{p,q}^{\kappa}(2-\eta)\big\{1+([2]_{p,q}-1)\tau\big\}(1+\tau)^{\xi}} \ z \right\vert ^{\gamma}d\theta.
\end{equation}

Thus, by applying Lemma \ref{l1}, it would suffice to show that
\begin{equation} \label{eq2}
1-\sum\limits_{\nu=2}^{\infty }c_{\nu}z^{\nu-1}\prec 1-\frac{1-\eta}{[2]_{p,q}^{\kappa}(2-\eta)\big\{1+([2]_{p,q}-1)\tau\big\}(1+\tau)^{\xi}} \ z.
\end{equation}

If the subordination (\ref{eq2}) holds true, then there exists an analytic function $\omega $ with $\omega (0)=0$ and $\left\vert \omega (z)\right\vert <1$ such that
\begin{equation} \label{eq3}
1-\sum\limits_{\nu=2}^{\infty }c_{\nu}z^{\nu-1} = 1-\frac{1-\eta}{[2]_{p,q}^{\kappa}(2-\eta)\big\{1+([2]_{p,q}-1)\tau\big\}(1+\tau)^{\xi}} \ \omega(z).
\end{equation}

Using Theorem \ref{thm21}, we have
\begin{eqnarray*}
\left\vert \omega (z)\right\vert &=& \left\vert \sum\limits_{\nu=2}^{\infty }\frac{[2]_{p,q}^{\kappa}(2-\eta)\big\{1+([2]_{p,q}-1)\tau\big\}(1+\tau)^{\xi}}{1-\eta }c_{\nu}z^{\nu-1}\right\vert \\
&\leq& \left\vert z\right\vert \sum\limits_{\nu=2}^{\infty } \frac{[\nu]_{p,q}^{\kappa}(\nu-\eta)\big\{1+([\nu]_{p,q}-1)\tau\big\}[1+(\nu-1)\tau ]^{\xi}%
}{1-\eta } \ c_{\nu} \leq \left\vert z\right\vert <1,
\end{eqnarray*}%
which proves the subordination (\ref{eq2}). This completes the proof of Theorem \ref{im}.
\end{proof}
%***********************************************************************************************************************************************************************************

\section{Inclusion relations involving neighborhoods}
The concept of neighborhoods was first initiated by Goodman \cite{good}, and then generalized by Ruscheweyh \cite{Rush}. In this section, we will investigate the $(\nu,\delta)$-neighborhoods of the subclass $\mathcal{P}_{p,q}^{\xi ,\kappa}(\tau, \eta)$.

First, we define $(\nu,\delta)$-neighborhoods of the function $\varphi\in\mathcal{T}$ by
\begin{equation} \label{eq4}
N_{\nu,\delta}(\varphi) = \left\{ \psi\in\mathcal{T}:\psi(z)=z-\sum\limits_{\nu=2}^{\infty} d_{\nu} z^{\nu} \ \text{and } \ \sum\limits_{\nu=2}^{\infty}\nu\left\vert c_{\nu}-d_{\nu}\right\vert \leq \delta \right\}.
\end{equation}

In particular, for the identity function $e(z) = z$, we have
\begin{equation} \label{eq5}
N_{\nu,\delta}(e) = \left\{ \psi\in\mathcal{T}:\psi(z)=z-\sum\limits_{\nu=2}^{\infty} d_{\nu} z^{\nu} \ \text{and } \ \sum\limits_{\nu=2}^{\infty}\nu d_{\nu} \leq \delta \right\}.
\end{equation}

Our first inclusion relation is given by
\begin{theorem}
\label{t2} If
\begin{equation} \label{teq3}
\delta =\frac{2(1-\eta )}{[2]_{p,q}^{\kappa}(2-\eta)\big\{1+([2]_{p,q}-1)\tau\big\}(1+\tau)^{\xi}},
\end{equation}
then
\begin{equation} \label{teq4}
\mathcal{P}_{p,q}^{\xi ,\kappa}(\tau, \eta) \subset N_{\nu,\delta }(e).
\end{equation}

\end{theorem}

\begin{proof}
Let the function $\varphi(z)$ be in the class $\mathcal{P}_{p,q}^{\xi ,\kappa}(\tau, \eta)$. Then Theorem \ref{thm21} yields
\begin{equation} \label{eq6}
[2]_{p,q}^{\kappa}(2-\eta)\big\{1+([2]_{p,q}-1)\tau\big\}(1+\tau)^{\xi}\sum\limits_{\nu=2}^{\infty } c_{\nu} \leq 1-\eta,
\end{equation}
which implies
\begin{equation} \label{eq7}
\sum\limits_{\nu=2}^{\infty } c_{\nu} \leq \frac{1-\eta}{[2]_{p,q}^{\kappa}(2-\eta)\big\{1+([2]_{p,q}-1)\tau\big\}(1+\tau)^{\xi}}.
\end{equation}

On the other hand, from (\ref{theq21}) and (\ref{eq7}), we have
\begin{eqnarray*} \label{eq8}
(1-\eta)[2]_{p,q}^{\kappa}\big\{1&+&([2]_{p,q}-1)\tau\big\}(1+\tau)^{\xi}\sum\limits_{\nu=2}^{\infty } \nu c_{\nu} \\
&\leq& 1-\eta-\eta[2]_{p,q}^{\kappa}\big\{1+([2]_{p,q}-1)\tau\big\}(1+\tau)^{\xi}\sum\limits_{\nu=2}^{\infty } c_{\nu} \\
&\leq& \frac{2(1-\eta)^2}{2-\eta}.
\end{eqnarray*}

Hence,
\begin{eqnarray*} \label{eq9}
\sum\limits_{\nu=2}^{\infty } \nu c_{\nu} \leq \frac{2(1-\eta)}{[2]_{p,q}^{\kappa}(2-\eta)\big\{1+([2]_{p,q}-1)\tau\big\}(1+\tau)^{\xi}}=\delta,
\end{eqnarray*}
which, by the definition (\ref{eq5}), confirms the inclusion (\ref{teq4}) asserted by Theorem \ref{t2}.
\end{proof}

Next, we determine the neighborhood for the class $\mathcal{P}_{p,q}^{\xi ,\kappa}(\tau, \eta)$, which we define as follows. \\
A function $\varphi\in\mathcal{T}$ is said to be in the class $\mathcal{P}_{p,q}^{\xi ,\kappa}(\tau, \eta)$ if there exists a function $\psi \in\mathcal{P}_{p,q}^{\xi ,\kappa}(\tau, \eta)$ such that
\begin{eqnarray*} \label{eq10}
\left\vert \frac{\varphi(z)}{\psi(z)} -1\right\vert < 1-\varepsilon, \ \ \ \ (0\leq\varepsilon <1; z\in\Omega).
\end{eqnarray*}

\begin{theorem}
\label{t3} If $\psi \in\mathcal{P}_{p,q}^{\xi ,\kappa}(\tau, \eta)$ and
\begin{equation} \label{teq5}
\varepsilon =1-\frac{\delta \ [2]_{p,q}^{\kappa}(2-\eta)\big\{1+([2]_{p,q}-1)\tau\big\}(1+\tau)^{\xi}}{2\Big([2]_{p,q}^{\kappa}(2-\eta)\big\{1+([2]_{p,q}-1)\tau\big\}(1+\tau)^{\xi}+\eta-1\Big)},
\end{equation}
then
\begin{equation} \label{teq6}
N_{\nu,\delta }(\psi) \subset \mathcal{P}_{p,q}^{\xi ,\kappa}(\tau, \eta).
\end{equation}

\end{theorem}

\begin{proof}
Let the function $\varphi(z)$ be in $N_{\nu,\delta }(\psi)$. Then definition \ref{eq4} yields
\begin{equation} \label{eq11}
\sum\limits_{\nu=2}^{\infty}\nu\left\vert c_{\nu}-b_{\nu}\right\vert \leq \delta,
\end{equation}
which implies that
\begin{equation} \label{eq12}
\sum\limits_{\nu=2}^{\infty}\left\vert c_{\nu}-d_{\nu}\right\vert \leq \frac{\delta}{2}.
\end{equation}

Next, since $\psi \in\mathcal{P}_{p,q}^{\xi ,\kappa}(\tau, \eta)$, we have [cf. equation \ref{eq7})]
\begin{equation} \label{eq13}
\sum\limits_{\nu=2}^{\infty } d_{\nu} \leq \frac{1-\eta}{[2]_{p,q}^{\kappa}(2-\eta)\big\{1+([2]_{p,q}-1)\tau\big\}(1+\tau)^{\xi}},
\end{equation}

Letting $|z|\rightarrow 1$, so that
\begin{eqnarray*} \label{eq14}
\left\vert \frac{\varphi(z)}{\psi(z)} -1\right\vert &\leq& \frac{\sum\limits_{\nu=2}^{\infty}\left\vert c_{\nu}-d_{\nu}\right\vert}{1-\sum\limits_{\nu=2}^{\infty}\left\vert d_{\nu}\right\vert } \\
&\leq& \frac{\delta}{2} \left(\frac{[2]_{p,q}^{\kappa}(2-\eta)\big\{1+([2]_{p,q}-1)\tau\big\}(1+\tau)^{\xi}}{[2]_{p,q}^{\kappa}(2-\eta)\big\{1+([2]_{p,q}-1)\tau\big\}(1+\tau)^{\xi}+\eta-1} \right) \\
&\leq& 1-\varepsilon,
\end{eqnarray*}
provided that $\varepsilon$ is given completely by (\ref{teq5}). Thus by the above definition, $\varphi \in\mathcal{P}_{p,q}^{\xi ,\kappa}(\tau, \eta)$, which completes the proof.
\end{proof}

%%%%%%%%%%%%%%%%%%%%%%%%%%%%%%%%%%%%%%%%%%%%%%%%%%%%%%%%%%%%%%%%%%%%%%%%%%%%%%%%%%%%%%%%%%%%%%%%%%%%%%%%%%%%%%%%%%%%%%%%%%%%%%%%%%%%%%%%%%%%%%%%%%%%%%%%%%%%%%%%%%%%%%%%%%%%%%%%%%%%%%%%%%%%%%%%%%%%%%%%
%%%%%%%%%%%%%%%%%%%%%%%%%%%%%%%%%%%%%%%%%%%%%%%%%%%%%%%%%%%%%%%%%%%%%%%%%%%%%%%%%%%%%%%%%%%%%%%%%%%%%%%%%%%%%%%%%%%%%%%%%%%%%%%%%%%%%%%%%%%%%%%%%%%%%%%%%%%%%%%%%%%%%%%%%%%%%%%%%%%%%%%%%%%%%%%%%%%%%%%%
                                %%%%%%%%%%%%%%%%%%%%%%%%%%%%%%%%%%%%%%%%%%%%%%%%%%%%%%%%%%%%%%%%%%%%%%%%%%%%%%%%%%%%%%%%%%%%%%%%%%%%%%%%%%%%%%%%%%%%%%%%%%%%%%%%%%%%%%%%
                                %%%%%%%%%%%%%%%%%%%%%%                                       Bibliography                                     %%%%%%%%%%%%%%%%%%%%%%%%%%
                                %%%%%%%%%%%%%%%%%%%%%%%%%%%%%%%%%%%%%%%%%%%%%%%%%%%%%%%%%%%%%%%%%%%%%%%%%%%%%%%%%%%%%%%%%%%%%%%%%%%%%%%%%%%%%%%%%%%%%%%%%%%%%%%%%%%%%%%%
%%%%%%%%%%%%%%%%%%%%%%%%%%%%%%%%%%%%%%%%%%%%%%%%%%%%%%%%%%%%%%%%%%%%%%%%%%%%%%%%%%%%%%%%%%%%%%%%%%%%%%%%%%%%%%%%%%%%%%%%%%%%%%%%%%%%%%%%%%%%%%%%%%%%%%%%%%%%%%%%%%%%%%%%%%%%%%%%%%%%%%%%%%%%%%%%%%%%%%%%

%%%%%%%%%%%%%%%%%%%%%%%%%%%%%%%%%%%%%%%%%%%%%%%%%%%%%%%%%%%%%%%%%%%%%%%%%%%%%%%%%%%%%%%%%%%%%%%%%%%%%%%%%%%%%%%%%%%%%%%%%%%%%%%%%%%%%%%%%%%%%%%%%%%%%%%%%%%%%%%%%%%%%%%%%%%%%%%%%%%%%%%%%%%%%%%%%%%%%%%%

\end{document}